\documentclass[a4paper]{amsart}
\usepackage{amsmath}
  \usepackage{graphics} 
\usepackage{graphicx}

\usepackage{color}
 
\usepackage{amsmath,amsfonts,amscd,bezier}
\usepackage{latexsym,amssymb}

\newcommand{\N}{\mathbb{N}}
\newcommand{\Z}{\mathbb{Z}}
\newcommand{\C}{\mathbb{C}}
\newcommand{\R}{\mathbb{R}}
\newcommand{\Q}{\mathbb{Q}}

\newcommand{\ee}{\mathrm{e}}

\newcommand{\va}{\varphi}

\DeclareMathOperator{\tr}{tr}
\DeclareMathOperator{\Fix}{Fix}

\newcommand{\dpt}{\displaystyle}

\newtheorem{thm}{Theorem}[section]
\newtheorem{prop}[thm]{Proposition}
\newtheorem{cor}[thm]{Corollary}
\newtheorem{lemma}[thm]{Lemma}
\newtheorem{defn}[thm]{Definition}

\begin{document}


\title[Recognition of symmetries in reversible maps]{Recognition of symmetries in reversible maps}
\author[P. H. Baptistelli, I. S. Labouriau \and M. Manoel]{ Patr\'icia H. Baptistelli, Isabel S. Labouriau \and Miriam Manoel}

\maketitle
\bigskip


\centerline{\scshape Patr\'icia H. Baptistelli\footnote{Email address: phbaptistelli@uem.br}}
\medskip

{\footnotesize \centerline{Departamento de Matem\'atica, CCE} \centerline{Universidade Estadual de
Maring\'{a}} \centerline{Av. Colombo, 5790, 87020-900}
\centerline{ Maring\'{a}, Brazil }}
\medskip

\centerline{\scshape Isabel S. Labouriau\footnote{Email address:
islabour@fc.up.pt}}
\medskip

{\footnotesize \centerline{Centro de Matem\'atica}
\centerline{Universidade do Porto} \centerline{Rua do Campo Alegre, 687, 4169-007} \centerline{Porto, Portugal}}
\medskip

\centerline{\scshape Miriam Manoel\footnote{Email address:
miriam@icmc.usp.br}}
\medskip

{\footnotesize \centerline{Departamento de Matem\'atica, ICMC}
\centerline{Universidade de S\~ao Paulo} \centerline{13560-970 Caixa
Postal 668,} \centerline{S\~ao Carlos, Brazil }}

\quad

\begin{abstract}
We deal with  germs of diffeomorphisms that are reversible  under an involution. 
We establish that this condition implies that, in general,
 both the  family of  reversing symmetries and the group of symmetries are not finite, in contrast with  continuous-time dynamics, where typically there are finitely many  reversing symmetries. From this we obtain two chains of fixed-points subspaces of involutory reversing symmetries that we use  to obtain geometric information on the discrete dynamics generated by a given   diffeomorphism. The results are 
illustrated by the generic case in arbitrary dimension, when the diffeomorphism is the composition of transversal linear involutions. 
\end{abstract}


\quad



\section{Introduction} \label{sec:introduction}

Symmetry occurs in many different contexts. It has driven attention  in many fields of  Mathematics and related areas whenever existence and analysis of patterns become relevant. 
Symmetric objects have characteristic features that are not present in generic objects.
Symmetry carries geometric information that facilitates the study of such objects.
One particularly important  kind of symmetry  -- or reversing symmetry, as we shall say -- is given by an involution. For a local study in $\R^n$, this is defined as the germ of a diffeomorphism  $\va: (\R^n,0) \to (\R^n,0)$ satisfying $\varphi \circ \varphi = I_n ,$ namely a germ of diffeomorphism which is its own inverse. We point out that many results in the present work are algebraic, so it shall be clear that those also hold for general invertible maps.
For the same reason, although formulated for diffeomorphisms, our results often do not require differentiability and hold for homeomorphisms.

One branch of applications of reversing symmetries comes from
dynamical systems.
We refer to \cite{LR} for references on reversibility and related problems. 
For discrete dynamics governed by
the iteration of
 a (germ of) diffeomorphism $F: (\R^n,0) \to (\R^n,0)$, we recall that this is called reversible under a (germ of) diffeomorphism  $R: (\R^n,0) \to (\R^n,0)$, or simply $R$-reversible, if $F \circ R = R \circ F^{-1}$. 
In other words, the symmetric copy of a trajectory is also a trajectory with time reversed.
If $F$ happens to be reversible under an involution $\va_1$, it follows that it can be decomposed as  $F=\va_1\circ\va_2$, for some involution $\va_2$. Yet,  $F$ is also $\va_2$-reversible. 
As a consequence, contrary to the continuous case, reversible discrete dynamical systems always have more than one reversing involution. 
We remark that there exist reversing symmetries  that are not involutions, even for linear isomorphisms; for example, $2\times 2$
 symmetric matrices 
 are reversible under the rotation of $\pi/2$. Here we assume throughout that $F$ has an involutory reversibility. In this case, interesting dynamics resides in the class governed by diffeomorphisms that possess an infinite number of involutions. In algebraic terms, if $\va_1$ and $\va_2$  generate an Abelian group, then these are simultaneously linearizable, as a consequence of the Bochner-Montgomery theorem about linearization of a compact group of transformations around a fixed point (see \cite{MZ}). In our setting, the two involutions generate a discrete group which may be non-Abelian and is generally noncompact. Then, a natural question raised in this set-up regards the local linearization around a fixed point. The simultaneous linearization and transversality of $\va_1$ and $\va_2$ is a sufficient condition for the linearization of $F$ (see \cite{MMT} for more on this topic). Also, pairs of involutions are linearizable provided $F$ is a hyperbolic germ of diffeomorphism \cite{Teixeira2}. The work in \cite{MMT2} extends this result to $F$  normally hyperbolic. 
In the first case, the fixed-point set of $F$ reduces to a point; in the second case, it can be a local submanifold with positive dimension. 
Normally hyperbolic examples are treated in \cite{BLM3}.

In the present work, simultaneously linearizable involutions are treated as a particular case of our  results.
Sets of linear involutions are treated in Section~\ref{subsecN2} for the planar case and in Section~\ref{subsecN3} for higher dimensions. 

Another branch of applications comes from Singularity theory in the study of divergent diagrams of folds (see, for example, \cite{MMT, Teixeira1, Voronin}).
We recall that  given a fold $f: (\R^n,0) \to (\R^n,0)$, there exists a unique nontrivial involution $\va$ associated with $f$, that is, such that $f \circ \va = f$.  
In \cite{MMT} the authors prove that equivalent classes of $s$-tuples of divergent diagram of folds, for $s \geq 1$ finite, are described through the simultaneous equivalence of the associated $s$-tuples of involutions. Here we remark that the same holds for an infinite sequence of folds. In addition, it follows that the study of  singular sets of folds is deduced directly from the description  of fixed-point subspaces of involutions we present here, once the singular set of  $f$ is precisely the fixed-point submanifold of the associated involution $\va$. Further investigation on this topic
is carried out in \cite{BLM2}.
 
Here is what we shall encounter in the next sections. In Section~\ref{secReversible} 
we recall some basic definitions and properties; we also present the main
general 
 results regarding symmetries and reversing symmetries of a reversible germ of diffeomorphism $F$. 
 We establish in Theorem~\ref{thm:chains} the existence of a chain of fixed-point subspaces of the (infinite) sequence of reversing involutions which must be tracked by the iterates of $F$. 
 In the particular case when the fixed-point submanifolds of $\va_1$ and $\va_2$ have codimension 1, Theorem~\ref{thm:FixPhin2} describes the orbits of points in the complement in $\R^n$ of these  fixed-point submanifolds.
 The connected components of this complement are interchanged by $F$.
 We also use the chain of fixed-point subspaces to find periodic points. In Section~\ref{subsecN2} the results of the Section~\ref{secReversible} are applied to the class of pairs of  transversal linear involutions on the plane. 
 This is done 
using the normal forms  obtained in \cite{MMT} under the equivalence given by simultaneous conjugacy.
We also relate the reversible dynamics to 
 the geometry of the chain of fixed-point subspaces of the reversing involutions.  In Section~\ref{subsecN3} we use the results of  Section~\ref{subsecN2}  
 to extend the analysis to dimension greater or equal to three. Here we use again normal forms of transversal linear involutions given in \cite{MMT} which, for almost all cases, are  suspensions of the normal forms on the plane. 

In short,  Section~\ref{secReversible} contains general results on reversible  diffeomorphisms.
These results are illustrated by detailed descriptions of the geometry in two dimensions (Section~\ref{subsecN2}) and in higher dimensions (Section~\ref{subsecN3}).

\section{Reversibility and equivariance} \label{secReversible}

The  local study developed in this paper is assumed to be about the origin $0 \in \R^n$.  Hence, we shall work with the notion of a germ: Let $F: U_{1} \subseteq \R^n \to \R^n$ be a mapping defined on a neighborhood $U_{1}$ of $0$ such that $F(0) = 0$. The germ of $F$ at $0$, denoted by $F: (\R^n, 0) \rightarrow (\R^n, 0)$, is the set of mappings $G:U_{2} \subseteq \R^n \to \R^n$ such that there is a neighborhood $U \subseteq U_{1} \cap U_{2}$ of $0$ with $F|_{U} = G|_{U}$. When $F$ is a diffeomorphism, we say that $F: (\R^n, 0) \rightarrow (\R^n, 0)$ is a germ of a diffeomorphism.

Let $\Omega$ be the group under composition of the invertible maps on $(\R^n,0)$ 
acting on $(\R^n,0)$ by the standard action of application,
\[  \Omega \times (\R^n,0) \to (\R^n,0),   \ \ \ (\varphi, x)  \mapsto \varphi x =  \va (x).  \]

\begin{defn} 
Let $\varphi \in \Omega.$ A germ of  a diffeomorphism  $F: (\R^n,0) \to (\R^n,0)$ is {\em $\varphi$-equivariant} if $F \circ \varphi = \varphi \circ F$. It is {\em $\varphi$-reversible} if $F \circ \varphi = \varphi \circ F^{-1}.$  In the first case, $\va$ is a {\em symmetry} of $F$ and in the second case $\va$ is a {\em reversing symmetry} of $F$.
\end{defn}

Note that if $F$ is $\va$-reversible then so is $F^{-1}.$ For a given germ of diffeomorphism $F$ of $(\R^n,0) $ we denote by $\Gamma_{+}$ the group formed by the symmetries of $F$ and by $\Gamma_{-}$ the set of reversing symmetries of $F,$ that is,
\begin{equation}
\label{group}
\Gamma_{+} = \{\va \in \Omega: F \circ \va = \va \circ F\} \quad \text{and} \quad \Gamma_{-} = \{\va \in \Omega : F \circ \va = \va \circ F^{-1}\}.
\end{equation}

%

\noindent  In general, the set $\Gamma_{-}$ doesn't have a group structure. Indeed, $\Gamma_{+} \cap \Gamma_{-} \neq \emptyset$ if, and only if,  $\Gamma_{+} = \Gamma_{-},$  which is equivalent to $F^2 = I_n$, where $I_n$ denotes the germ of the identity map on $\R^n$. This is the only case for which $\Gamma_{-}$ is a group. For  $F^2 \neq I_n$, we have that $\Gamma_{-}$ is closed under inversion, but composition of two reversing symmetries belongs to $ \Gamma_{+}.$ Moreover, we can write $\Gamma_{-} = \delta \Gamma_{+}$ for any $\delta \in \Gamma_{-}$ fixed (and arbitrary). For $\va \in \Gamma_{+}$ we have 
 $$
 F \circ (\delta \circ \va)  = \delta \circ F^{-1} \circ \va  = (\delta \circ \va) \circ F^{-1},
 $$ 
hence
 $\delta \circ \va \in \Gamma_{-}.$ Conversely, if $\va \in \Gamma_{-}$ then $\va = \delta \circ (\delta^{-1} \circ \va) \in \delta \Gamma_{+}.$ 
 
\begin{defn}
An {\em involution} is a germ of a diffeomorphism $\va: (\R^n,0) \to (\R^n,0)$ satisfying $\varphi \circ \varphi = I_n.$
\end{defn}

From now, we assume that there exists an involution in $\Gamma_{-}.$  In \cite{MMT} it has been recognized that, in this case, reversible diffeomorphisms are in 1--1 correspondence with 
 pairs of involutions. In fact, if $\va_1$ is an involution then $F$ is $\varphi_1$-reversible if, and only if, $F = \varphi_1 \circ \varphi_2$ for 
the involution $\va_2=\va_1\circ F$.
We now also remark that, in this case, $F$ is $\varphi_2$-reversible too, for
$$
F \circ \va_2 = \va_1 = \va_1^{-1} = \va_2 \circ \va_2^{-1} \circ \va_1^{-1} = \va_2 \circ F^{-1}.
$$ 
Hence   $F$  corresponds to the
 pair of reversing symmetries $(\va_1, \va_2)$. 
In what follows we show that, more than one pair,  there are actually  two infinite sequences of involutions $\va_k$ and $\va_k'$
such that $F$ is $\va_k$- and $\va_k'$-reversible if $F^{m} \neq I_n, $ for all  $m \in {\Bbb N}$. 
We take 
\begin{equation} \label{eq:sequences}
\va_k=\va_2\circ F^{k-2}, \  \  \va_k'=F^{k-1}\circ\va_1, \ \ k\in\N, \ \  k\ge 1.
\end{equation}
This definition is consistent for $k = 1$ and $k = 2,$ with $\va_1' = \va_1.$

\begin{prop}
Let $\va_1: (\R^n,0) \to (\R^n,0)$ be an involution and let  $F: (\R^n,0) \to (\R^n,0)$ be  a $\va_1$-reversible germ of  diffeomorphism such that $F^{m} \neq I_n, $ for all $m \in  {\Bbb N}$. Then $F$ has an infinite group of symmetries and an infinite set of reversing symmetries.

\end{prop}

\begin{proof}
$F$ is clearly a symmetry of $F$ itself, so the
subgroup  generated by $F$,  
$$
[ F ] = \{ F^{k} :  k \in {\Bbb Z}\},
$$ 
is formed by symmetries, and then $\Gamma_{+}$ is infinite.
Furthermore, all $\va_k$ and $\va_k'$ defined in (\ref{eq:sequences}) are 
different elements
in $ \Gamma_{-},$ so $\Gamma_{-}$ is infinite.
\end{proof}

Examples are given in  Sections~\ref{subsecN2} and \ref{subsecN3}.

\begin{defn}
Given a germ of a diffeomorphism $F: (\R^n,0) \to (\R^n,0)$, the {\em $F$-orbit} of a point $x\in (\R^n,0)$ is the ordered set $\{x_k =  F^k(x) : \ k\in\Z\}$. When $F$ is clear from the context we just call this set the {\em orbit} of $x$.
\end{defn}

If $F: (\R^n,0) \to (\R^n,0)$ is the germ of a $\va$-reversible diffeomorphism, then $\va$ maps the $F$-orbit of $x \in (\R^n,0)$ into the $F^{-1}$-orbit of $\va(x)$, preserving the order.

\subsubsection{The fixed-point sets}
\begin{defn}
The {\em fixed-point set} of  a map-germ $\va \in \Omega$  is $$\Fix(\va)=\left\{ x\in (\R^n,0):\ \va(x)=x \right\}$$ and the {\em fixed-point set} of a subgroup $\Sigma\le \Omega$ is $$\Fix(\Sigma)=\left\{ x\in (\R^n,0) :\ \gamma x=x, \ \forall  \ \gamma\in\Sigma \right\}.$$
\end{defn}

 If $\va$ is linear and $\Sigma$ is a subgroup of the linear group ${\rm GL}(n)$  then $\Fix(\va)$ and $\Fix(\Sigma)$ are naturally extended to the whole $\R^n$ as linear subspaces of $\R^n$. 
 It follows that  the fixed-point set of  any
involution $\va: (\R^n,0) \to (\R^n,0)$ is a smooth submanifold in 
 $(\R^n,0)$, since $\va$ is conjugate to the germ of its linear part $d \va(0)$ at the origin (see \cite[Lemma 2.2]{MMT2}). 
 Denote by $\langle a_1, \ldots, a_{\ell} \rangle$ the linear subspace generated by $a_1, \ldots, a_\ell$.  
 
The following result is classical and extensively used in equivariant 
conti\-nuous-time
dynamics:
 \begin{lemma}
Let $F: (\R^n,0) \to (\R^n,0)$ be a germ of  an equivariant diffeomorphism  with symmetry group $\Gamma_+$.
If  $\Sigma\le\Gamma_+$ is a subgroup, then $\Fix(\Sigma)$ is $F$-invariant.
\end{lemma}

\begin{proof}
A point $x \in (\R^n,0)$ belongs to $\Fix(\Sigma)$ if, and only if, $\gamma x=x$ for all $\gamma\in\Sigma$. Then if $F$ is equivariant, we have $\gamma F(x)=F(\gamma x)= F(x)$ for all $\gamma\in\Sigma$. Hence $F(x)\in \Fix(\Sigma)$.
\end{proof}

 Fixed-point sets of symmetries of $F$ are therefore invariant under the discrete dynamics ruled by $F$. There is no similar result for $\Gamma_{-}$. Firstly, subsets of $\Gamma_{-}$ do not have a group structure. In addition, if $\Sigma$ is a subset of $\Gamma_{-}$ and $x \in {\rm Fix}(\Sigma)$, then $\gamma F^{-1}(x)=F(\gamma x)= F(x)$ for all $\gamma\in\Sigma$. Concerning reversing symmetries, we have:

\begin{lemma}\label{lemma:FixPhin}
Let $\va_1: (\R^n,0) \to (\R^n,0)$ be an involution and let  $F: (\R^n,0) \to (\R^n,0)$ be a $\va_1$-reversible germ of diffeomorphism. 
Consider the two sequences of reversing symmetries of $F$ given in (\ref{eq:sequences}).  The following equalities hold:
$$ \label{eq:equalities}
F\left( \Fix(\va_{k+2})\right) = \Fix(\va_k),   \ \  F\left( \Fix(\va_{k}')\right) = \Fix(\va_{k+2}'),  \ \ k \in {\N},  \ \ k \geq 1. 
$$
\end{lemma}
\begin{proof}
Consider $x \in \Fix(\va_{k+2})$, {\it i.e.},   $x= \va_2(F^{k}(x))$.
If $y=F(x)$, then
$$
y=F(\va_2(F^{k}(x)))=\va_2(F^{-1}(F^{k}(x)))=\va_2(F^{k-2}(F(x)))=
\va_k(y),
$$
so $F\left( \Fix(\va_{k+2})\right)\subseteq \Fix(\va_k)$.

For the other inclusion, let  $y \in \Fix(\va_k).$ 
Then $$y = \va_k(y) =   \va_2 \circ F^{k-2}(y) = F (\va_2 \circ F^{k-1}(y))$$ and, therefore, $y = F(x)$ for  $x=\va_2 \circ F^{k-1}(y)$. Also,  $\va_{k+2} (x) =\va_2\circ F^{k}(x)=\va_2\circ F^{k-1}(y)=x$.

The equalities for the other sequence are obtained analogously. 
\end{proof}

\begin{thm} \label{thm:chains}
 Applying $F$ to the fixed-point submanifolds of the involutions of (\ref{eq:sequences})  the following chains are obtained:

\[  \cdots \stackrel{}{\longrightarrow}  \Fix(\va_{2k}) 
\stackrel{}{\longrightarrow} \cdots   \stackrel{}{\longrightarrow}   \Fix(\va_{2})     \stackrel{}{\longrightarrow} \Fix(\va_{2}')  \stackrel{}{\longrightarrow}  
\cdots \stackrel{}{\longrightarrow}  \Fix(\va_{2k}') \stackrel{}{\longrightarrow}  \cdots, \]
\begin{equation}  \label{eq:chain}
\end{equation}
\[  \cdots \stackrel{}{\longrightarrow}  \Fix(\va_{2k+1}) 
\stackrel{}{\longrightarrow} \cdots   \stackrel{}{\longrightarrow}   \Fix(\va_{1})     \stackrel{}{\longrightarrow} \Fix(\va_{3}')  \stackrel{}{\longrightarrow}  
\cdots \stackrel{}{\longrightarrow}  \Fix(\va_{2k+1}') \stackrel{}{\longrightarrow}  \cdots, \]
for $k \geq 1.$ 
\end{thm}
\begin{proof}
Similar calculations  to those of
Lemma~\ref{lemma:FixPhin} give
 \[F(\Fix(\va_1) ) = \Fix(\va_3') \quad \mbox{and} \quad F(\Fix(\va_2) ) = \Fix(\va_2').  \]
We now use Lemma~\ref{lemma:FixPhin} to get the result.
\end{proof}

If for any $k\ge 1$ and $\ell\in\N$ we have either $\Fix(\va_k)=\Fix(\va_{k+2\ell})$ or  $\Fix(\va_k)=\Fix(\va'_{k+2\ell})$  or $\Fix(\va_k')=\Fix(\va_{k+2\ell}')$  then the whole chain  in \eqref{eq:chain}
containing one of these fixed-point manifolds is finite.

The following is also a direct consequence of Lemma~\ref{lemma:FixPhin}:

\begin{cor} \label{cor:dimension}
All fixed-point submanifolds of the involutions of (\ref{eq:sequences}) with odd index 
have dimension equal to $\dim \Fix(\va_1)$, and the ones for even index  have dimension equal to $\dim \Fix(\va_2)$.
\end{cor}

It should be stressed that dynamically and geometrically relevant results should not depend on the choice of coordinates. Proposition~\ref{prop:equivalence} below establishes this point. Before stating it we define the equivalence of two sets of involutions which is given by simultaneous conjugacy:

\begin{defn}
Two pairs $(\va_1, \va_2)$ and $(\psi_1, \psi_2)$ of involutions on $(\R^n,0)$ are {\em equiva\-lent} if there exists a germ of diffeomorphism $h: (\R^n,0) \to (\R^n,0)$ such that $\psi_i = h \circ \va_i \circ h^{-1},$ for $i = 1,2.$  If  
$h$ is a germ of a linear isomorphism,  we say that $(\va_1, \va_2)$ and $(\psi_1, \psi_2)$ are {\em linearly equivalent}.
\end{defn}
 
For two equivalent pairs of involutions $(\varphi_1, \varphi_2)$ and $(\psi_1, \psi_2)$, it follows that  for each $k \in {\N}$, $k \geq 2,$ the two pairs $(\va_k, \va_k')$ and $(\psi_k, \psi_k'),$ constructed in  
(\ref{eq:sequences}) for $F = \varphi_1 \circ \varphi_2$ and $G = \psi_1 \circ \psi_2$ respectively, are  equivalent. In addition, $F$ and
$G$ are also conjugate, so they generate equivalent dynamics. These equivalences are clearly governed by the same $h$, which directly implies the following:

\begin{prop} \label{prop:equivalence}
The equalities in Lemma~\ref{lemma:FixPhin} and the chains in Theorem~\ref{thm:chains} are invariant under equivalence.
\end{prop}

The next result gives a sufficient condition for an orbit starting at a fixed-point submanifold to be periodic:

\begin{prop}
For any reversing symmetry $\varphi$ of  $F$,  if $x \in \Fix(\va)$ and if there exists $\ell \in {\N}$ such that $F^\ell(x) \in \Fix(\va)$, then the orbit of $x$ is periodic. In addition, for the sequences of reversing symmetries of $F$ in 
(\ref{eq:sequences}) and $\ell > k,$ we have:
\begin{enumerate}
\renewcommand{\theenumi}{(\Alph{enumi})}
\renewcommand{\labelenumi}{{\theenumi}}
\item\label{caso1}
 If $x \in \Fix(\va_k) \cap \Fix(\va_\ell),$ then the orbit of $x$ is a periodic orbit with period that divides $\ell - k.$ Also, if $x$ is a periodic point with period that divides $\ell - k$ and $x \in \Fix(\va_k),$ then $x \in \Fix(\va_\ell).$
\item\label{caso2}
 If  $x \in \Fix(\va_k') \cap \Fix(\va_\ell'),$ then the orbit of $x$ is a periodic orbit with period that divides $\ell- k$. Also, if $x$ is a periodic point with period that divides $\ell - k$ and $x \in \Fix(\va_k'),$ then $x \in \Fix(\va_\ell').$
\item\label{caso3}
  If $x \in \Fix(\va_k') \cap \Fix(\va_\ell)$, then the orbit of $x$ is a periodic orbit with period that divides $k+ \ell -2$. Also, if $x$ is a periodic point with period that divides $k + \ell - 2$ and $x \in \Fix(\va_k'),$ then $x \in \Fix(\va_\ell).$
\end{enumerate} 
\end{prop}
\begin{proof}
The first part  is straightforward, just noticing that $x, F^{\ell}(x)  \in \Fix(\va)$ implies that 
$$ F^{-\ell}(x) =  F^{-\ell}( \va(x))  = \va (F^{\ell}(x)) = F^\ell(x). $$
For the statements \ref{caso1}--\ref{caso3} we use \eqref{eq:sequences} to get
$$
\va_k \circ \va_\ell = F^{\ell - k},\quad 
\va_k' \circ \va_\ell' = F^{k-\ell },\quad
\va_k' \circ \va_\ell = F^{\ell + k-2}.
$$
The periodicity then follows from 
$\Fix(\va_k) \cap \Fix(\va_\ell) \subseteq \Fix(\va_k \circ \va_\ell)=\Fix(F^{\ell - k})$ for \ref{caso1} and similarly for \ref{caso2} and for \ref{caso3}.

\end{proof}

A particular case of the proposition above can be found in \cite{Devaney}, namely when the whole space has even dimension and the fixed-point subspaces are  $n/2$-dimensional submanifolds of $\R^n$.

From  Corollary~\ref{cor:dimension}, if $\dim \Fix(\va_1) = \dim \Fix(\va_2) = n-1$,  then the  fixed-point submanifolds split $({\Bbb R}^n,0)$ into connected regions. The result below describes how the dynamics by $F = \va_1 \circ \va_2$  behaves with respect to these regions. 

\begin{thm}\label{thm:FixPhin2}
Let $\va_1, \va_2: (\R^n,0) \to (\R^n,0)$ be two involutions with $\dim \Fix(\va_1) = \dim \Fix(\va_2) = n-1$. Let $\va_k$, $\va_k'\in\Gamma_-$, $k\in\N$, $k\ge 1$ be as in (\ref{eq:sequences}).  
Then  $F = \va_1 \circ \va_2$ interchanges the connected components of the germ at the origin of
$$
\mathcal{C} = \R^n\setminus \bigcup_{k=1}^\infty\left( \Fix(\va_k)\cup  \Fix(\va_k')\right), $$
determined by these fixed-point submanifolds. 

\end{thm} 

\begin{proof}
 Take $V$ a region whose boundary is determined by the fixed-point manifolds of involutions $\psi_p, \psi_q \in  \Delta = \{\va_k, \va_k': k \geq1 \}.$ By path connectedness of $F(V)$ and  Lemma~\ref{lemma:FixPhin},  the boundary of $F(V)$ is determined by  $\Fix(\widehat{\psi}_p) \cup \Fix(\widehat{\psi}_q)$, with $\widehat{\psi}_p, \widehat{\psi}_q \in \Delta $ 
 where each pair $(\psi_p, \widehat{\psi}_p)$   and $(\psi_q, \widehat{\psi}_q)$ consists of consecutive elements in one  of the chains in (\ref{eq:chain}), that is, 
 \[ F(\Fix(\psi_p)) = \Fix(\widehat{\psi}_p), \ F(\Fix(\psi_q)) = \Fix(\widehat{\psi}_q). \] 
 Therefore, $F(V)$ is another component. 
 
 It remains to consider the case when part of the boundary of a connected component $V$ of the complement
 is not contained in a fixed-point manifold of any involution in $\Delta$ (see the examples of Subsections~\ref{subs:trace2}, \ref{subs:trace>2} and \ref{subsec:last}). This happens when the boundary of $V$ meets the set ${\mathcal{C}}$ of  accumulation points of these fixed-point manifolds.
  We claim that $F({\mathcal{C}})\subset{\mathcal{C}}$.
  Indeed, a point $x\in{\mathcal{C}}$
 may be written as $\dpt x=\lim_{n\to\infty} x_n$ with $x_n\in \Fix(\psi_n)$ for some $\psi_n\in\Delta$.
 Since $F$ is a homeomorphism, then $\dpt F(x)=\lim_{n\to\infty} F( x_n)$. 
 Lemma~\ref{lemma:FixPhin} implies that $F( x_n)\in \Fix(\widehat{\psi}_n)$ for some $\widehat{\psi}_n\in\Delta$,
 establishing the claim. It follows that if the boundary of a component $V \subset {\mathcal{C}}$ consists of accumulation points of fixed-point subspaces, then $V$ is mapped by $F$ into a component with the same type of boundary. The same is true if the boundary of $V$ contains both elements of fixed-point submanifolds and elements of ${\mathcal{C}}$, completing the proof.
\end{proof}

We now define transversality of two involutions. This is a generic condition we assume for the pairs of involutions treated in the next sections.

\begin{defn} \label{def:transversal}
Two involutions $\va_1, \va_2$ on $(\R^n, 0),$ $n \geq 2,$ are {\em transversal} if $\Fix (\va_1)$ and $\Fix (\va_2)$ are in general position at 0, i.e., $$\R^n = T_0 \Fix (\va_1) + T_0 \Fix (\va_2),$$ where $T_0 \Fix(\va_i)$ denotes the tangent space to $\Fix(\va_i)$ at 0, $i = 1,2$.
\end{defn}

In the next two sections we apply the previous results to analyse the behavior of a $\va_1$-reversible germ of diffeomorphism $F$ associated with a pair $(\va_1, \va_2)$ of transversal linear involutions on $(\R ^n, 0),$ for $n=2$ and $n \geq 3,$ respectively. For this, we restrict our study to the linear case, considering the group of symmetries $\Gamma_{+}$ and the set of reversing symmetries $\Gamma_{-},$ defined in (\ref{group}), as subsets of the linear group ${\rm GL}(n)$ and keeping the notation introduced in this section.  


\section{Dynamics and geometry of linear reversible maps on the plane} \label{subsecN2}

In this section we consider a germ of diffeomorphism $F = \psi_1 \circ \psi_2$ on $(\R^2, 0),$ where $\psi_1$ and $\psi_2$ are transversal linear involutions (Definition~\ref{def:transversal}). We present the results up to equivalence of pairs of involutions given by simultaneous conjugacy. Hence, the pair $(\psi_1,\psi_2)$ is considered to be in normal form, which is given in \cite[Theorem 6.2]{MMT}. For that, we first recall the definition of the antipodal subspace of a linear involution $\va$ on $\R^n,$
$$\mathcal{A}(\va) = \{x \in \R^n: \  \va(x) = -x\}.$$ Notice that $\R^n = \Fix(\va) \oplus \mathcal{A}(\va).$ Denote by $\Lambda = [\psi_1, \psi_2]$ the group generated by $\psi_1$ and $\psi_2.$ There are three cases to be considered: 
\begin{enumerate}
\renewcommand{\theenumi}{(\roman{enumi})}
\renewcommand{\labelenumi}{{\theenumi}}
\item\label{d2i}
$\Lambda$ is Abelian;
\item\label{d2ii}
 $\Lambda$ is non-Abelian and $\mathcal{A}(\psi_2) = \Fix (\psi_1);$ 
\item\label{d2iii} $\Lambda$ is non-Abelian and $\mathcal{A}(\psi_2) \neq \Fix (\psi_1).$ 
  \end{enumerate} 
 Our aim is to investigate, for the cases above, the fixed-point subspaces of the reversing involutions in $\Lambda_{-} = \Lambda \cap \Gamma_{-}$ and their relation to the dynamics generated by $F.$ Let us denote by $\Lambda_{+} = \Lambda \cap \Gamma_{+}$ the group generated by $F$ and by ${\bf Z}_2(\va)$ the 2-element group generated by $\va.$ 

\subsection{$\Lambda$ is Abelian} \label{subsec:Abelian}

This is a trivial case. By \cite[Theorem 6.2]{MMT}, $(\psi_1,\psi_2)$ is equivalent to $(\va_1,\va_2),$ where 
$$\va_1(x,y) = (-x,y) \quad \mbox{and} \quad \va_2 (x,y) = (x,-y),$$ so $\Lambda = {\bf Z}_2(\va_1) \oplus {\bf Z}_2(\va_2).$ Since $F = -I_2,$ the dynamics is rather degenerate because all the $F$-orbits (except the origin) are periodic of  period 2. Moreover, $$\va_{2k+1} = \va_{2k+1}' = \va_1 \quad \mbox{and} \quad \va_{2k} = \va_{2k}' = \va_2$$ for all $k \geq 1.$ Therefore, $\Fix(F) = \{(0,0)\}$ and the fixed-point subspaces $\Fix(\va_1)= \langle (0,1) \rangle$ and $\Fix(\va_2)= \langle (1,0) \rangle$ divide the plane into four connected components that are interchanged by $F$ (Theorem~\ref{thm:FixPhin2}). 

Moreover, $\Gamma_{+} = \Gamma_{-} = {\rm GL}(2),$ while $\Lambda_{+} = {\bf Z}_2(-I_2)$ and $\Lambda_{-} = \{\va_1,\va_2\}.$

\subsection{$\Lambda$ is non-Abelian and $\mathcal{A}(\psi_2) = \Fix (\psi_1)$} \label{subsec:Non1}

By \cite[Theorem 6.2]{MMT}, $(\psi_1,\psi_2)$ is equivalent to $(\va_1,\va_2),$ where $$\va_1(x,y) = (-x, x+y) \quad \mbox{and} \quad \va_2(x,y) = (x,-y).$$ In this case $F(x,y) = (-x,x-y),$ whose eigenvalues are $\lambda_1 = \lambda_2 = -1$ with geometric multiplicity 1. Writing $F = -I_2 + N,$ where  $$N = \begin{pmatrix}
0 & 0 \\ 1 & 0 \end{pmatrix},$$ we have $F^k = (-1)^k (I_2 - kN) \neq I_2$ for all $k \in \mathbb{N},$ 
which implies that $$\va_k(x,y) = (-1)^k(x,(k-2)x - y) \quad \text{and} \quad \va_k'(x,y) = (-1)^k(x,-kx - y).$$ 
 Yet, $$\Fix(F) = \{(0,0)\} =  \Fix(F^{2k+1}), \quad \Fix(\va_{2k}) = \langle(1,k-1)\rangle, \quad \Fix(\va_{2k}') = \langle(1,-k)\rangle,$$ $$\Fix(F^{2k}) = \Fix(\va_{2k+1}) = \Fix(\va_{2k+1}') = \langle(0,1)\rangle,$$ for all $k \geq 1.$ Therefore, the fixed-point subspaces of $\va_{2k}$ and $\va_{2k}'$ approach the $y$-axis as $k$ tends to infinity (see Figure \ref{figFix}). 

For the dynamics, we use the expression of $F^k$ to conclude that the $y$-axis is $F$-invariant and that the orbits of all its points, except the origin, have period 2. For the other points, by linearity, it suffices to look at the orbits of points $(1,y)$ given by
 $F^k(1,y)=(-1)^k\left(1,y-k\right) $, as illustrated in Figure~\ref{figFix}, on the right. This case provides an interesting illustration of Theorem~\ref{thm:FixPhin2}. For instance, the sector $\{(x,y): \ x>0,\ 0<y<x\}$ between $\Fix(\va_2)$ and $\Fix(\va_4)$ is mapped onto the sector
 $\{(x,y): \ x<0,\ 0<y<-x\}$ between $\Fix(\va_2')$ and $\Fix(\va_2)$.
Also,  the sector $\{(x,y): \ x<0,\ -x<y<-2x\}$ between $\Fix(\va_2')$ and $\Fix(\va_4')$ is mapped onto the sector
 $\{(x,y): \ x>0,\ -3x<y<-2x\}$ between $\Fix(\va_6')$ and $\Fix(\va_4')$. See Figure~\ref{figFix} on the left. 

Symmetries and reversing symmetries are as follows:
$$\Lambda_{+} = \left\lbrace (-1)^k (I_2 - kN): k \in \Z\right\rbrace , \quad \Lambda_{-} = \left\lbrace \va_k,\ \va_k': k \in \N, \ k \geq 1\right\rbrace ,$$ 
$$\Gamma_{+} = \biggl\{ \begin{pmatrix}
a & 0 \\ c & a
\end{pmatrix}: \ a, c \in \R, \ a \neq 0 \biggr\} \quad \text{and} \quad \Gamma_{-} = \biggl\{ \begin{pmatrix}
a & 0 \\ c & -a
\end{pmatrix}: \ a, c \in \R, \ a \neq 0 \biggr\},$$ both $\Gamma_{\pm}$ manifolds of dimension 2.

\begin{figure}
\begin{center}
\includegraphics[height=6cm]{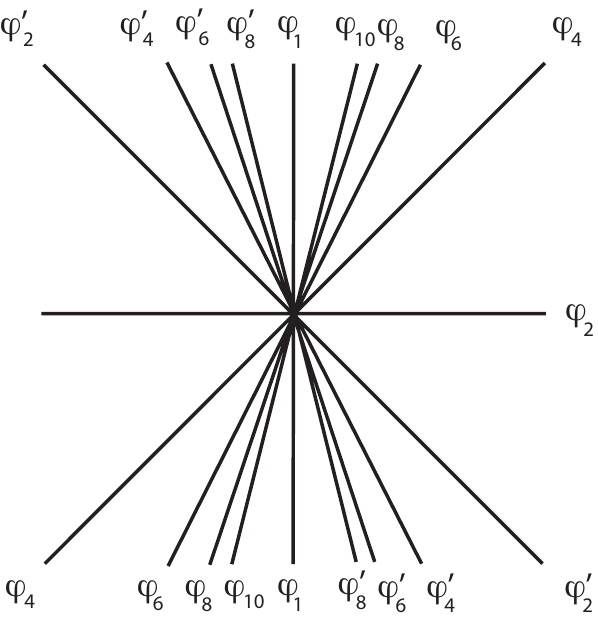}
\qquad
\includegraphics[height=6cm]{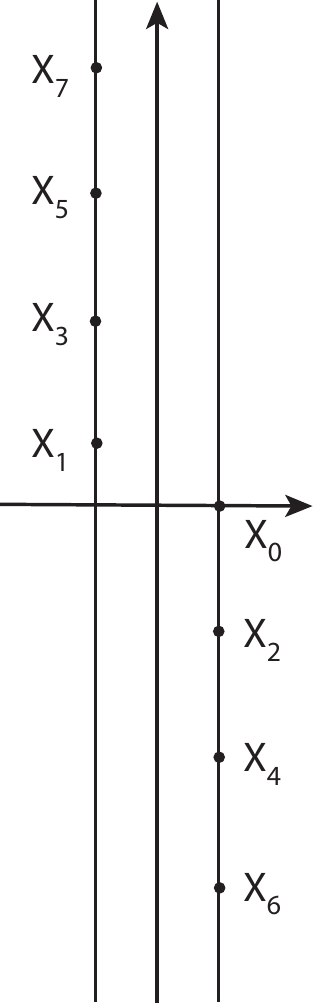}
\end{center}
\caption{Left: fixed-point subspaces for the involutions $\va_1$, $\va_{2k}$ and $\va_{2k}'$ in case~\ref{d2ii}.
Note that $\Fix(\va_{2k+1}')=\Fix(\va_{2k+1})=\Fix(\va_1)$, $k=1,2,3,\ldots$.
Right: orbit of the point $X_0=(1,0)$, $X_k=F^k(X_0)$.
}
\label{figFix}
\end{figure}

\subsection{$\Lambda$ is non-Abelian and $\mathcal{A}(\psi_2) \neq \Fix (\psi_1)$} \label{subsec:Non2}
Let us denote $t = \tr(\psi_1 \circ \psi_2).$ From \cite[Theorem 6.2]{MMT}, $(\psi_1,\psi_2)$ is equivalent to $(\va_1,\va_2),$ where 
\begin{equation}
\label{eq:case(iii)} \va_1(x,y) = (-x, y + (2 + t)x) \quad \text{and} \quad \va_2(x,y) = (x+y,-y).
\end{equation}
The analysis here considers all possibilities for the parameter $t,$ which is an invariant under linear simultaneous conjugacy. We have 
$$
F = \begin{pmatrix}
-1 & -1 \\ 2 + t & 1 + t
\end{pmatrix},$$ whose eigenvalues  
$$
\lambda_+ = \dfrac{t+ \sqrt{t^2 - 4}}{2}
 \qquad \text{and} \qquad 
 \lambda_- = \dfrac{t - \sqrt{t^2 - 4}}{2}
$$  satisfy $\lambda_{+}\lambda_{-} = 1.$ Notice that $\lambda_+ = \lambda_- = \pm 1$ if, and only if, $t = \pm 2,$ respectively. A direct calculation gives the group of symmetries of $F$
$$
\Gamma_{+} = \biggl\{ \begin{pmatrix}
a+(2+t) b & \ b \\ -(2+t)b & \ a
\end{pmatrix} \in {\rm GL}(2): \ a, b \in \R   \biggr\} 
$$
and the set of reversing symmetries of $F$ 
$$
 \Gamma_{-} = 
 \begin{pmatrix}
1 & 1 \\ 0& -1
\end{pmatrix} 
\Gamma_{+}
 =
 \biggl\{ \begin{pmatrix}
a & \ a+ b \\ (2+t) b & \ -a
\end{pmatrix} \in {\rm GL}(2): \ a, b \in \R\biggr\},
$$ both manifolds of dimension 2.


Since the normal form of $F$ depends only on the parameter $t,$ we have subdivided this subsection in four cases. In all of them the group of symmetries and the set of reversing symmetries gene\-rated by $\va_1$ and $\va_2$ are given respectively by $\Lambda_{+} = \left\lbrace F^k: k \in \Z\right\rbrace$ and $\Lambda_{-} = \left\lbrace \va_k,\ \va_k': k \in \N, \ k \geq 1\right\rbrace$.

\subsubsection{Normal form (\ref{eq:case(iii)}) with $t=- 2$} \label{subs:trace-2} In this case, $F(x,y) = (-x - y,-y),$ whose eigenvalues are $\lambda_1 = \lambda_2 = -1$ with geometric multiplicity 1. Writing $F = -I_2 + N,$ where $N$ is a nilpotent matrix of index 2, we have $F^k(x,y) = (-1)^k(x + ky, y)$ for all $k \in \mathbb{Z},$ which implies that $$\va_k(x,y) = (-1)^k(x + (k-1)y,-y) \quad \text{and} \quad \va_k'(x,y) = (-1)^k(x - (k-1)y,- y).$$ Therefore $$\Fix(F) = \{(0,0)\} = \Fix(F^{2k+1}), \quad \Fix(\va_{2k}) = \Fix(\va_{2k}') = \Fix(F^{2k}) = \langle (1,0) \rangle,$$ $$\Fix(\va_{2k+1}) = \langle(-k,1)\rangle \quad \mbox{and} \quad \Fix(\va_{2k+1}') = \langle(k,1)\rangle$$ for all $k \geq 1.$ The fixed-point subspaces of $\va_{2k+1}$ and $\va_{2k+1}'$ approach the $x$-axis as $k$ tends to infinity, which is an $F$-invariant line (see Figure \ref{figFixC2}).


\subsubsection{Normal form (\ref{eq:case(iii)}) with t = 2}  \label{subs:trace2} In this case, $F(x,y) = (-x - y,4x + 3y),$ whose eigenvalues are $\lambda_1 = \lambda_2 = 1$ with geometric multiplicity 1. Writing $F = I_2 + N$ for a nilpotent matrix $N$ of index 2, we have $$F^k (x,y) = \big((1-2k)x - ky, 4kx + (1+2k)y\big), \ \forall \ k \in \mathbb{Z},$$ which implies that $\va_k(x,y) = ((2k-3)x + (k-1)y,-4(k-2)x - (2k-3)y)$ and $\va_k'(x,y) = ((1-2k)x - (k-1)y,4kx - (1-2k)y).$ Therefore $$\Fix(F) = \Fix(F^{k}) = \langle (1,-2)\rangle$$ is the eigenspace of $F$ associated with $\lambda = 1$ for all $k \in \Z$ non-zero. 
Moreover, $\Fix(\va_{k}) = \langle(k-1,4-2k)\rangle$ and $\Fix(\va_{k}') = \langle (k-1,-2k) \rangle$ for all $k \geq 1.$ The fixed-point subspaces of $\va_{k}$ and $\va_{k}'$ approach the $F$-invariant line $y = -2x$ as $k$ tends to infinity (see Figure \ref{figFixC2}). Here the two half-lines $$\Fix(F) \cap \{(x,y) \in \R^2: \ x > 0\} \quad \text{and} \quad \Fix(F) \cap \{(x,y) \in \R^2: \ x < 0\}$$ are components of $\dpt \R^2\setminus \bigcup_{k=1}^\infty\left( \Fix(\va_k)\cup  \Fix(\va_k')\right).$


\begin{figure}
\begin{center}
\includegraphics[height=5.4cm]{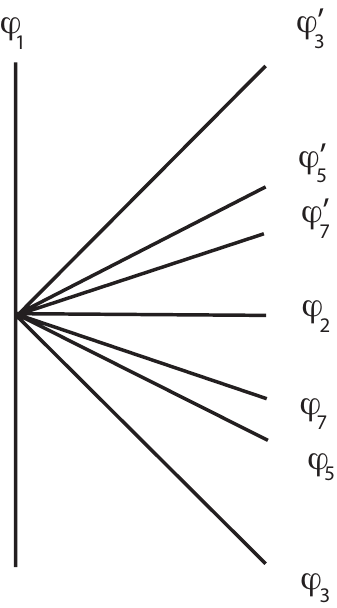}
\qquad\qquad
\includegraphics[height=4cm]{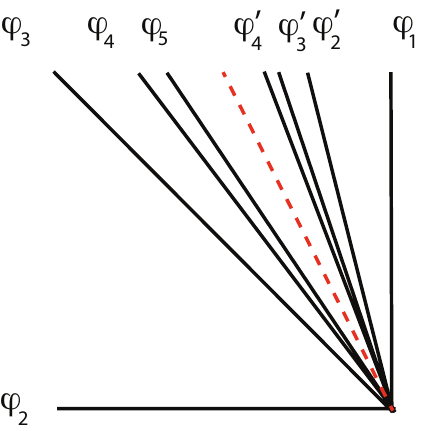}
\end{center}
\caption{Fixed-point subspaces for the involutions $\va_1$, $\va_{2k}$ and $\va_{2k}'$ in case~\ref{d2iii}
with $t=\pm 2$. 
Left: $t=  -2$. In this case
 $\Fix(\va_{2k}')=\Fix(\va_{2k})=\Fix(\va_2)$, $k = 1,2,3,\ldots$.
Right: $t=  2$.  The dashed red line is $\Fix(F)=\{(x,y): y=-2x\}$ where all the lines $\Fix(\va_k)$ and $\Fix(\va_k')$ accumulate as $k\to\infty$.
}
\label{figFixC2}
\end{figure}

\subsubsection{Normal form (\ref{eq:case(iii)}) with $\vert t \vert < 2$}  \label{subs:trace<2} When $-2<t<2$, the map $F$ has complex eigenvalues $$\lambda_\pm= \dfrac{t\pm i \sqrt{4-t^2}}{2},$$ with $|\lambda_\pm|=1.$ Hence $F$ is diagonalizable  over $\C$
 with 
$\lambda_\pm=\ee^{\pm i\theta},$ where $\theta=\arccos(t/2).$ This means that there is a change of coordinates that conjugates $F$  to a rotation of $\theta$ radians around the origin. The complex eigenvectors associated to $\lambda_{\pm}$ have the form $R+iI$ where
$$
R=\left(1,-\dfrac{t}{2} - 1\right)\qquad
I=\left(0,-\dfrac{\sqrt{4-t^2}}{2} \right).
$$
From now on we use coordinates in the basis $\beta=\{R,I\}$ of $\R^2$  for which, taking $\alpha_k=(k-1)\theta$, with $k \geq 1$, we have
$$
\left. \va_k\right|_\beta=\begin{pmatrix}
-\cos\alpha_k& \sin\alpha_k\\ \sin\alpha_k& \cos\alpha_k
\end{pmatrix}
\quad \mbox{and} \quad
\left. \va_k'\right|_\beta=\begin{pmatrix}
-\cos\alpha_k & -\sin\alpha_k \\ -\sin\alpha_k &\cos\alpha_k
\end{pmatrix},
$$
whence  $$\Fix(\va_k) = \langle\left.\left(\sin\alpha_k,1+\cos\alpha_k\right)\right|_\beta \rangle \quad \mbox{and} \quad \Fix(\va_k') = \langle \left.\left(-\sin\alpha_k,1+\cos\alpha_k\right)\right|_\beta \rangle.$$
Hence the coordinates in the basis $\beta$ of the generators of $\Fix(\va_k)$ and $\Fix(\va_k')$
lie on a circle of center $(0,1)$ and radius 1 (see Figure~\ref{FigFixIrrational}). 
When $\theta/2\pi= p/q$, $p, q \in \mathbb{Z}$, $q\ne 0$,  then $F^q=I_2$  and all the $F$-orbits are periodic of  period $q.$ Also, $\Lambda_{-}$ is finite and their fixed-point subspaces form a finite set of lines through the origin. When $\theta /2\pi \notin \Q,$ $\Lambda_{-}$ is infinite and a set of generators of $\Fix(\va_k)$ and of $\Fix(\va_k')$ can be taken to form each a dense set in the circle, and hence the union of  fixed-point subspaces is dense in the plane. In the original coordinates, there is a family of concentric $F$-invariant ellipses and each $F$-orbit is dense on the ellipse that contains it.

\begin{figure}
\begin{center}
\includegraphics[height=5cm]{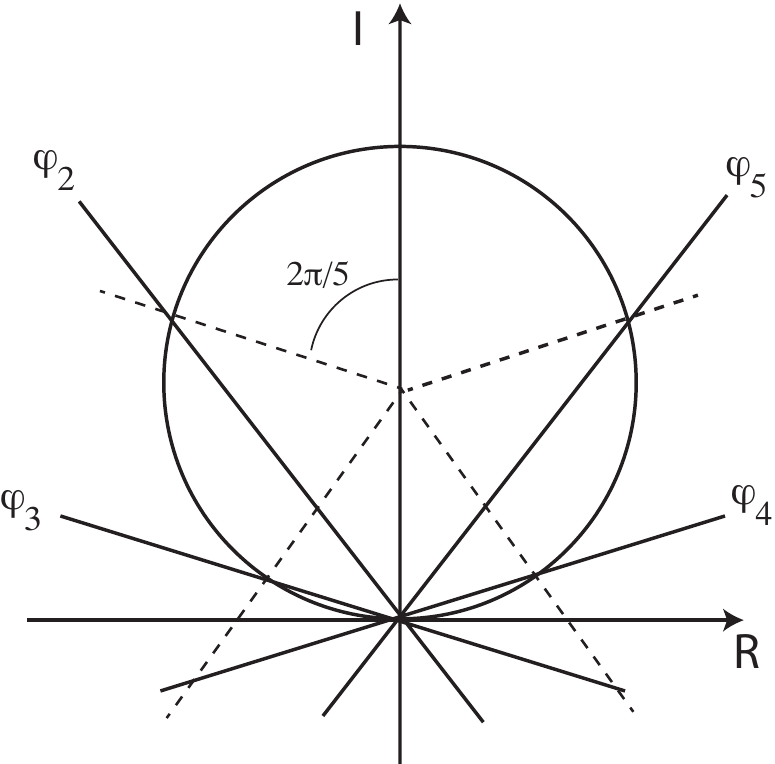}
\end{center}
\caption{Fixed-point subspaces in the $R\times I$ coordinates for the involutions $\va_k$ in case~\ref{d2iii}, when $|t|<2$ for $\theta=2\pi/5$. 
In this case $\Fix(\va_1)$ is the $I$-axis,  $\Fix(\va_{k})=\Fix(\va_{l})=\Fix(\va_{r}')$ when $k\equiv l\pmod{5}$ and $r+k\equiv 2\pmod{5}$.
}
\label{FigFixIrrational}
\end{figure}

%
%

\bigbreak

\subsubsection{Normal form (\ref{eq:case(iii)}) with $\vert t \vert > 2$}  \label{subs:trace>2} When $t>2$ the map $F$ has real eigenvalues $\lambda_+>1$ and $0<\lambda_-<1$, whereas for $t<-2$ the eigenvalues of $F$ satisfy $\lambda_- <-1<\lambda_+<0$. 
Hence, $F$ is hyperbolic and $F^k\ne I_2$ for $k\ne 0$. The eigenvectors of $F$ associated to $\lambda_{\pm}$ are generated by $(1,-1-\lambda_\pm),$ respectively. From now on we use coordinates in the basis $\beta=\{(1,-1-\lambda_+),(1,-1-\lambda_{-})\}$ of $\R^2$  for which we have $F$ diagonal,  
$$
\left. \va_k\right|_\beta=\begin{pmatrix}
0 & -\lambda_{-}^{k-1}\\ -\lambda_{+}^{k-1} & 0
\end{pmatrix}
\quad \mbox{and} \quad
\left. \va_k'\right|_\beta=\begin{pmatrix}
0 & -\lambda_{+}^{k-1}\\ -\lambda_{-}^{k-1} & 0
\end{pmatrix}.
$$ Therefore, $$\Fix(\va_k) = \langle (1,-\lambda_{+}^{k-1})|_\beta \rangle \quad \mbox{and} \quad \Fix(\va_k') =  \langle(1,-\lambda_{-}^{k-1})|_\beta\rangle.$$
These fixed-point subspaces do not coincide with the eigenspaces of $F$.
Hence, powers of $F$ and $F^{-1}$ map  $\Fix(\va_k)$, $\Fix(\va_k')$ into distinct subspaces, according to Theorem \ref{thm:chains}.
Moreover, it follows that the
  $\Fix(\va_k')$'s accumulate, when $k\to\infty$, on the expanding eigenspace of $F$: 
 if $t>2$  the $\Fix(\va_k')$'s accumulate  on the  eigenspace of $\lambda_+>1$ and for 
  $t<-2$  they accumulate on the eigenspace of $\lambda_-<-1$ (see Figure~\ref{figFixCT}).
Similarly, the subspaces  $\Fix(\va_k)$'s accumulate, when $k\to\infty$, on the contracting eigenspace of $F$ (the expanding eigenspace of $F^{-1}$).

\begin{figure}[h] 
\begin{center}
\includegraphics[height=6.5cm]{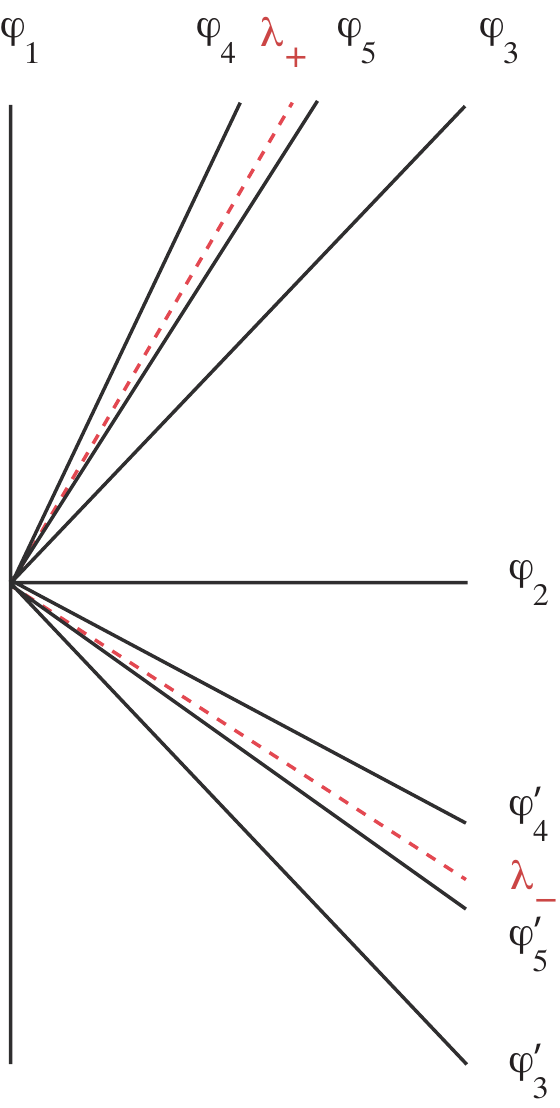}
\qquad\qquad
\includegraphics[height=5cm]{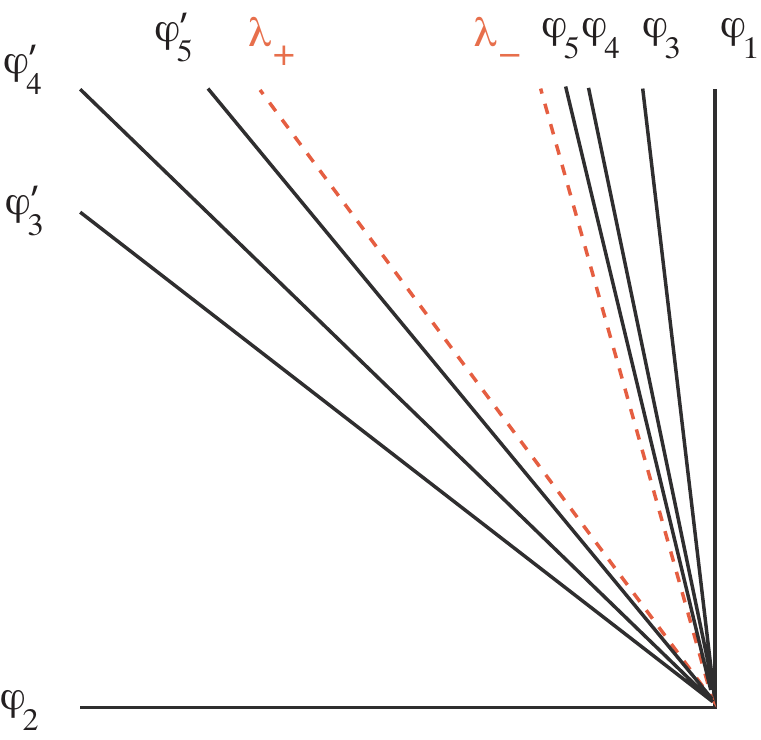}
\end{center}
\caption{Eigenspaces for $F$ (red dotted lines) and fixed-point subspaces for the involutions $\va_1$, $\va_{2k}$ and $\va_{2k}'$ in case~\ref{d2iii} for $|t|>2$.
Left: $t= -3$. 
Right: $t= 3$.  
}
\label{figFixCT}
\end{figure}

\section{Dynamics and geometry of linear reversible maps for $n \geq 3$} \label{subsecN3}

In this section we obtain a generalization of the results of Section~\ref{subsecN2} for $n \geq 3.$ As we shall see, the planar case leads to a similar analysis of the dynamics of a germ of diffeomorphism $F = \psi_1 \circ \psi_2$ on $(\R^n, 0),$ for $n \geq 3,$ where $\psi_1$ and $\psi_2$ are transversal linear involutions. Again, we denote by $\Lambda = [\psi_1, \psi_2]$ the group generated by the involutions and use the normal forms for the pairs of transversal linear involutions on $(\R^n, 0)$ now given in \cite[Theorem 7.3]{MMT}. There are five cases to consider:
\begin{enumerate}
\renewcommand{\theenumi}{(\alph{enumi})}
\renewcommand{\labelenumi}{{\theenumi}}
 \item\label{d3a}
 $\Lambda$ is Abelian;
  \item\label{d3b}
  $\Lambda$ is non-Abelian, $\tr(F) \neq n$ and $\mathcal{A}(\psi_2) \subset \Fix (\psi_1);$
   \item\label{d3c}
   $\Lambda$ is non-Abelian, $\tr(F) \neq n$ and $\mathcal{A}(\psi_2) \not\subset \Fix (\psi_1);$
  \item\label{d3d}
 $\Lambda$ is non-Abelian, $\tr(F) = n$ and $\mathcal{A}(\psi_1) = \mathcal{A}(\psi_2);$  
  \item\label{d3e}
  $\Lambda$ is non-Abelian, $\tr(F) = n$ and $\mathcal{A}(\psi_1) \neq \mathcal{A}(\psi_2).$ 
 \end{enumerate} 
 We investigate the relation between the fixed-point subspaces of the involutions in $\Lambda_{-}$ and the dynamics generated by $F.$ Let us denote by $e_i$ the vector with 1 in the $i$-th coordinate and $0$ elsewhere,  for $i=1, \ldots, n$.

 The following definition will be useful:
 
 \begin{defn} 
The map-germ   $\widehat{f}: (\R^{m+\ell},0) \to (\R^{m+\ell},0)$ is a {\em suspension}  of $f: (\R^{m},0) \to (\R^m,0)$ if 
$\widehat{f}(x,y)=(f(x),y),$ where $x\in\R^m$ and $y\in\R^\ell$.
The pair $(\widehat{\va}_1,\widehat{\va}_2)$ is a {\em suspension} of $(\va_1,\va_2)$ if each $\widehat{\va}_i$ is a suspension of $\va_i$ in the same system of coordinates.
\end{defn}

A consequence of the results of \cite{MMT} is that a pair of linear transversal involutions $(\psi_1,\psi_2)$ is equivalent to a suspension of a pair of planar involutions, except in case \ref{d3e}, with $n \geq 4$, for which the pair is equivalent to a suspension of a pair of involutions in $\R^3$.
The following trivial proposition summarises the properties of suspensions.

\begin{prop}\label{P42}
For linear involutions $\widehat{\va}_i: (\R^{m+\ell},0) \to (\R^{m+\ell},0)$, $i=1,2,$ if
$(\widehat{\va}_1,\widehat{\va}_2)$ is a suspension of the involutions $(\va_1,\va_2)$, with $\va_i:(\R^{m},0) \to (\R^m,0),$ then:
\begin{itemize}
\item $\widehat{F}=\widehat{\va}_1\circ\widehat{\va}_2$ is a suspension of $F=\va_1\circ\va_2$;
\item $\widehat{\va}_k$ and $\widehat{\va}'_k$, $k \geq 1$ integer, are suspensions of $\va_k$ and $\va_k'$ respectively;
\item $\Fix(\widehat{\va}_k)=\Fix(\va_k)\times\R^\ell$ and 
$\Fix(\widehat{\va}'_k)=\Fix(\va'_k)\times\R^\ell$;
\item The group $\widehat{\Gamma}_+$ of symmetries of $\widehat{F}$ consists of invertible  matrices of the form
$$\left( \begin{array}{c|c}
A & C \\ \hline D & B  
\end{array} \right),$$
where $A \in \Gamma_+$ is a symmetry of $F$, $B\in {\rm {GL}} (\ell)$ 
 with $FC=C$ and $DF=D$. A similar result holds for the set $\widehat{\Gamma}_-$  of reversing symmetries of $\widehat{F}$.
\end{itemize}
\end{prop}

Using \cite[Theorem 7.3]{MMT}, the relation of cases \ref{d3a}--\ref{d3d} above to \ref{d2i}--\ref{d2iii} of the previous section is the following:

\begin{itemize}
\item[\ref{d3a}]
if $\Lambda$ is Abelian, then $(\psi_1,\psi_2)$ is equivalent to $(\widehat{\va}_1,\widehat{\va}_2),$ where $(\widehat{\va}_1,\widehat{\va}_2)$ is a suspension of the normal forms in Subsection~\ref{subsec:Abelian};
 \item[\ref{d3b}]
  if $\Lambda$ is non-Abelian, $\tr(F) \neq n$ and $\mathcal{A}(\psi_2) \subset \Fix (\psi_1)$, then $(\psi_1,\psi_2)$ is equiva\-lent to $(\widehat{\va}_1,\widehat{\va}_2),$ where $(\widehat{\va}_1,\widehat{\va}_2)$ is a suspension of the normal forms in Subsection~\ref{subsec:Non1};
  \item[\ref{d3c}]
   if $\Lambda$ is non-Abelian, $\tr(F) \neq n$ and $\mathcal{A}(\psi_2) \not\subset \Fix (\psi_1)$, then $(\psi_1,\psi_2)$ is equiva\-lent to $(\widehat{\va}_1,\widehat{\va}_2),$ where $(\widehat{\va}_1,\widehat{\va}_2)$ is a suspension of the normal forms in Subsection~\ref{subsec:Non2} with $t\ne 2$;
 \item[\ref{d3d}]
 if $\Lambda$ is non-Abelian, $\tr(F) = n$ and $\mathcal{A}(\psi_1) = \mathcal{A}(\psi_2)$, then $(\psi_1,\psi_2)$ is equivalent to $(\widehat{\va}_1,\widehat{\va}_2),$ where $(\widehat{\va}_1,\widehat{\va}_2)$ is a suspension of the normal forms in Subsection~\ref{subsec:Non2} with $t=2$.
\end{itemize}

All these cases satisfy the hypothesis of Theorem~\ref{thm:FixPhin2}, {\it i.e.}, $\Fix(\psi_1)$ and $\Fix(\va_2)$ are hyperplanes.

In the next subsection we discuss  the remaining case \ref{d3e}, which does not suspend from the planar problem.

\subsection{Case~\ref{d3e}: $\Lambda$ is non-Abelian, $\tr(F) = n$ and $\mathcal{A}(\psi_1) \neq \mathcal{A}(\psi_2)$} \label{subsec:last}

By \cite[Theorem 7.3]{MMT}, $(\psi_1,\psi_2)$ is equivalent to $(\va_1,\va_2),$  where
\begin{itemize} \item[] $\va_1(x_1,\ldots, x_n) = (-x_1, 4x_1 + x_2, x_3, \ldots, x_n),$ \item[] $\va_2(x_1,\ldots, x_n) = (x_1 + x_2, -x_2, x_2 + x_3, x_4, \ldots, x_n).$ \end{itemize} In this case $F = I_n + N,$ where 
$$N = \left( \begin{array}{ccc| c}
-2 & -1 & 0 &  \\ 4 & 2 & 0 &   \\ 0 & 1 & 0 & \\ \hline &  & &  0  
\end{array} \right)$$ is a nilpotent matrix of index 3. Therefore $F$ has eigenvalues $\lambda = 1$ with algebraic multiplicity $n$ and geometric multiplicity $n-2.$ Since 
$$F^k = I_n + kN + \dfrac{k(k-1)}{2}N^2 = \left( \begin{array}{lll| c}
1-2k & -k & 0 &  \\ 4k & 2k + 1 & 0 &    \\ 2k(k-1) & k^2 & 1  &   \\ \hline &  &  & I_{n-3} \end{array} \right) $$ 
for each $k \in \Z,$ we have 
$$\Fix(F) = \Fix(F^{k}) = \langle e_3, \ldots, e_n\rangle,$$ for all $k \in \Z$ non-zero. Moreover, 
$$\Fix(\va_{k}) = \langle(k-1)e_1 + (4-2k)e_2, e_3, \ldots, e_n\rangle$$ 
and 
$$\Fix(\va_{k}') = \langle(k-1)e_1 -2ke_2, e_3, \ldots, e_n\rangle$$ for all $k \geq 2.$ For $k =1,$ we have  $\Fix(\va_{1}) = \Fix(\va_{1}') = \langle e_2, e_3, \ldots, e_n\rangle.$ 
Thus, when  $k \to \infty$, the fixed-point subspaces of $\va_{k}$ and $\va_{k}'$ approach the $F$-invariant subspace $x_2 = -2x_1$.

The subspaces $\Fix(\va_k)$ and $\Fix(\va_k')$ have codimension 1, so we can still apply  Theorem~\ref{thm:FixPhin2}.
The situation is similar to the example of Subsection~\ref{subs:trace2}: the invariant limit hyperplane contains two connected 
components of 
$$\dpt \R^n \setminus \bigcup_{k = 1}^{\infty}\left( \Fix(\va_k) \cup \Fix(\va_k')\right).$$

%
%

Linear symmetries of $F$ have the form given in Proposition~\ref{P42} and, where $A\in\Gamma_{+}$ and $ B\in { \rm { GL}}(n-3),$ with $FC=C$ and $DF=D$, where $\Gamma_{+}$ is the 3-dimensional manifold of elements
$$ \begin{pmatrix}
a & b & 0 \\ -4b & a - 4b & 0 \\ 2(b+c) & c & a - 2b
\end{pmatrix}, \ \ \ a, b, c \in \R, \ a \neq 2b.$$ 
 Linear reversing symmetries of $F$ have the same form, with $A\in \Gamma_{-} = \va_1 \Gamma_+$, the 3-dimensional manifold of elements 
$$ \begin{pmatrix}
-a & -b & 0 \\ 4(a-b) & a & 0 \\ 2(b+c) & c & a - 2b
\end{pmatrix}, \ \  \ a, b, c \in \R, \ a \neq 2b.$$

\bigbreak

\subsection*{Acknowledgments} 
%


CAPES/FCT  provided financial support  under the grant 88887.125430/2016-00 for  visits of the authors to the Universities of Porto and of S\~ao Paulo, whose hospitality is gratefully acknowledged.
I.S.L. had partial financial support through CMUP (UIDB/00144/2020), which is funded by Funda\c c\~ao para a Ci\^encia e a Tecnologia I.P. (Portugal).

\medbreak

\end{document}